\documentclass[11pt,reqno]{amsart}
\usepackage{latexsym}
\usepackage{amssymb}
\usepackage{amsmath}
\usepackage{amsthm}
\usepackage{nccmath}
\usepackage{amsfonts}
\usepackage{amssymb} 
\usepackage{booktabs}
\usepackage{rotating}
\usepackage{lscape}
\usepackage{color}
\usepackage{enumerate}
\usepackage{todonotes}

\numberwithin{equation}{section}

\newcommand{\hh}{\mathbb{H}}
\newcommand{\pp}{{\mathbb P}}

\newcommand{\RR}{\mathbb R}

\def\Z{Z\!\!\!Z}
\def\C{{\mathbb C}}
\def\t{\theta}
\def\tt#1#2{{\t\left[\begin{matrix}{#1}\\ {#2}\end{matrix}\right]}}

\newcommand{\be}{\begin{equation}}
\newcommand{\ee}{\end{equation}}
\newcommand{\bes}{\begin{equation*}}
\newcommand{\ees}{\end{equation*}}

\newcommand{\eqn}{\begin{eqnarray}}
\newcommand{\feqn}{\end{eqnarray}}
\newcommand{\eqnn}{\begin{eqnarray*}}
\newcommand{\feqnn}{\end{eqnarray*}}

\newtheorem{theorem}{Theorem}

\newtheorem{deff}{Definition}
\newtheorem{remark}{Remark}

\begin{document}

\title{A new proof of the Caporaso-Sernesi theorem via Weber's formula}

\author{Francesco Dalla Piazza}
\address{Dipartimento di Matematica, Universit\`a ``La Sapienza'', Piazzale A. Moro 2, I-00185, Roma, Italy}
\email{dallapiazza@mat.uniroma1.it, f.dallapiazza@gmail.com}

\author{Alessio Fiorentino}
\address{Dipartimento di Matematica, Universit\`a ``La Sapienza'', Piazzale A. Moro 2, I-00185, Roma, Italy}
\email{fiorentino@mat.uniroma1.it}
%


\begin{abstract}
In this paper we give a new proof of Caporaso and Sernesi's result which states that the general plane quartic is uniquely determined by its 28 bitangents. Our proof uses classical geometric results, as it is based on Weber's formula and on the injectivity of the $\theta^{(4)}$ map.
\end{abstract}

\maketitle

\section{Introduction}\noindent
It is classically known that the number of the bitangents to a non singular curve of degree $d$ in the projective plane is given by the formula $\frac{1}{2}d(d-2)(d^2-9)$. 
The first important result that relates the configurations of bitangents to the geometry of the curve is Arhonold's classical theorem which states that a smooth plane quartic can be recovered by the configuration of any of the $288$ $7$-tuples of bitangents such that the six contact points of any subtriple of bitangents do not belong to the same conic in the projective plane;
these $7$-tuples of bitangents are known as Aronhold systems. 
There have been several generalizations of this result over the decades. One of the most relevant is due to Caporaso and Sernesi \cite{CS}. Using GIT techniques, they actually improved it by proving the following theorem.
\begin{theorem}[Caporaso-Sernesi] \label{thm:CS}
The general plane quartic is uniquely determined by its 28 bitangent lines.
\end{theorem}
In \cite{lh} Lehavi proved that a non singular plane quartic can be reconstructed from its $28$ bitangents,  providing a method to derive an explicit formula for the curve. These results are stronger than Aronhold's one, because the knowledge of both the bitangents and their contact points on the curve is needed to get the configuration of the Aronhold systems, whereas the sole configuration of the bitangents is enough to describe the geometry of the plane quartic.
In this paper we give a new proof of Theorem \ref{thm:CS} using classical geometric results.

In general, a ratio of two theta constants can be written as a rational function of Jacobian determinants of gradients of odd theta functions, i.e.:
\be \label{ident}
\frac{\theta_{m_1}}{\theta_{m_2}}=\frac{p(D(n_1,\ldots,n_g))}{q(D(k_1,\ldots,k_g))}.
\ee
This is essentially due to the generical injectivity of the theta gradients map \cite{mainarticle}:
\begin{align*}
\mathcal{G}_g: \mathcal{A}_g^{4,8} &\longrightarrow \rm{Gr}_{\C}(g,2^{g-1}(2^g-1)) \\
\tau &\longrightarrow [\ldots ,{\rm grad}_z\theta_n|_{z=0}(\tau), \ldots]_{n\,{\rm odd}},
\end{align*}
whose Pl\"{u}cker coordinates are the Jacobian determinants $D(n_1,\ldots,n_g)$. Since the map is birational, the field of rational functions on the level moduli space $\mathcal{A}_g^{4,8}$ is isomorphic to $\mathbb{C}[\frac{D(n_1,\ldots,n_g)}{D(k_1,\ldots,k_g)}]$. On the other hand, the field of rational functions on $\mathcal{A}_g^{4,8}$ is also isomorphic to $\mathbb{C}[\frac{\theta_{m_1}}{\theta_{m_2}}]$ by the generical injectivity of the theta map:
\begin{align*}
\bar{\mathbb{H}_g}/\Gamma_g(4,8) &\longrightarrow \pp^{2^{g-1}(2^g+1)-1}\\
\tau &\longmapsto [\theta_{m_1}(\tau):\ldots :\theta_{m_{2^{g-1}(2^g+1)}}(\tau)].
\end{align*}
In genus three an explicit expression for the identity \eqref{ident} is due to Weber \cite{W}.

Our proof of Theorem \ref{thm:CS} is based on the  injectivity of the $\theta^{(4)}$ map \cite{SM} and on Weber's formula. The main point is that in Weber's formula the right side of \eqref{ident} is actually a ratio of monomials in the Jacobian determinants and the odd characteristics defining each Jacobian determinant appear with the same multiplicity in the numerator and in the denominator.


\section {Aknowledgments}
We are grateful to Riccardo Salvati Manni for explaining us this interesting topic and for several stimulating discussions and suggestions.

\section{Riemann theta functions and the Siegel modular group}
The tube domain $\hh_g$ of complex symmetric $g \times g$ matrices with positive definite imaginary part is known as the {\it Siegel upper half-plane of degree $g$}. A transitive action of the symplectic group $Sp(2g,\RR)$ is defined on $\hh_g$  by biholomorphic automorphisms:
\begin{align} \label{act}
\cdot:Sp(2g,\RR) \times \hh_g &\longrightarrow \hh_g \cr
\left ( \gamma , \tau \right ) &\longmapsto \gamma\cdot\tau:=(a \tau + b)(c \tau + d)^{-1},
\end{align}
 where the generic element of $Sp(2g,\RR)$ is conventionally written in a standard block notation as:
\bes
 \gamma = \begin{pmatrix} a & b \\ c & d \end{pmatrix}   \quad \quad \text{with} \, \, a,  b, c, d \, \, \text{real $g \times g$ matrices.} 
\ees
The {\it Siegel modular group} $\Gamma_g:= Sp(2g,\Z)$ on $\hh_g$ is a remarkable subgroup of $Sp(2g,\RR)$ because of its geometrical relevance. Since its action is properly discontinuous, the coset space $ \mathcal{A}_g := {\hh_g}/{\Gamma_g}$ is a normal analytic space (cf. \cite{123} and \cite{123o}), which is classically known to be isomorphic to the moduli space of principally polarized abelian varieties (p.p.a.v.), see \cite{D,GH} for details.
A {\it congruence subgroup} of the Siegel modular group $\Gamma_g$ is a subgroup $\Gamma$ containing for some $n \in \mathbb{N}$ the level subgroup:
\bes
\Gamma_g(n) = \{\, \gamma \in \Gamma_g \mid \gamma \equiv 1_{2g} \, \text{mod} \,  n  \,\}.
\ees The subgroups $\Gamma_g(n)$, which are of course the simplest examples of congruence subgroups, are normal in $\Gamma_g$.

For $m',m''\in \Z^g$ and $z\in \C^g$ the {\it Riemann theta function with characteristic} $m= \left[{}^{m'}_{m''} \right]$ is the series:
\bes
\theta_m(\tau,z):=\tt {m'}{m''}(\tau,z):=\sum\limits_{p\in\Z^g} \exp \left[\left(
p+\frac{m'}{2},\tau(p+\frac{m'}{2})\right)+2\left(p+\frac{m'}{2},z+
\frac{m''}{2}\right)\right],
\ees
where $\exp(z):=e^{\pi i z}$  and $(\cdot,\cdot)$ at the exponent stands for the standard inner product in the complex euclidean space. As a holomorphic function on $\hh_g \times \C^g$ the Riemann theta function with characteristic $m$ is characterized up to a constant factor by the equation:
\be
\label{signtheta}
\theta_{m+2n}(\tau , z) = (-1)^{{}^tm'n''} \theta_{m} (\tau , z )
\ee 
and by the heat equation (cf. \cite{Igusa3}).
By virtue of (\ref{signtheta}), Riemann theta functions are parametrized up to a sign by a {\it $g$-characteristic}, namely a column vector $\begin{bmatrix} m' \\ m'' \end{bmatrix}$ with $m', m'' \in \Z_{2}^g$. The set of $g$-characteristics will be conventionally denoted by the symbol $\mathcal{C}_g$. For each $m \in \mathcal{C}_g$ the function $\theta_m: \hh_g \rightarrow \C$ defined by $\theta_m (\tau) := \theta_m(\tau , 0)$ is known as the {\it theta constant} with $g$-characteristic $m$ (or simply with characteristic $m$, when there is no ambiguity). A parity function  $e(m):=(-1)^{^t m' m''}$ can be defined on the set $\mathcal{C}_g$ so as to classify the $2^{2g}$ $g$-characteristics into {\it even} and {\it odd} ones, respectively if $e(m)=1$ or $e(m)=-1$:
\bes
\begin{array}{lll}
E_g:=\{ m \in \mathcal{C}_g \mid e(m)=1\},& \rm{card}(E_g)=2^{g-1}(2^g+1),\\
\\
O_g:=\{ m \in \mathcal{C}_g \mid e(m)=-1\},& \rm{card}(O_g)=2^{g-1}(2^g-1).
\end{array}
\ees
Note that a theta constant $\theta_m$ is non vanishing if and only if $m \in E_g$. The subsets $E_g$ and $O_g$ can be actually regarded as the two orbits into which the set $\mathcal{C}_g$ decomposes under the action of the Siegel modular group $\Gamma_g$, defined as follows:
\begin{align} \label{action}
\cdot:\Gamma_g\times \mathcal{C}_g &\longrightarrow \mathcal{C}_g \cr
(\gamma,m) &\longmapsto \gamma \cdot\begin{bmatrix} m' \\ m'' \end{bmatrix} := \left [ \begin{pmatrix} d &  -c \\ - b &  a \end{pmatrix} \begin{pmatrix} m' \\ m'' \end{pmatrix} + \begin{pmatrix} diag(c{}^t  d ) \\ diag( a{}^t  b)  \end{pmatrix} \right ] \text{mod}\, 2. 
\end{align}
The behaviour of the Riemann theta functions under the actions in (\ref{act}) and (\ref{action}) is described by the so called \emph{transformation formula} (cf. \cite{Igusa3} and \cite{Igusa10}):
\begin{align} \label{eq:transth}
&\theta_{\gamma\cdot m}(\gamma\cdot\tau,{}^t(c\tau+d)^{-1}z)=\epsilon_\gamma(m)\Phi(m,\gamma,\tau,z){\rm det}(c\tau+d)^{1/2}\theta_m(\tau,z) \cr
&\forall\gamma\in\Gamma_g,\,\forall\tau\in\hh_g,\,\forall z\in \C^g,\, \forall m \in \mathcal{C}_g,
\end{align}
where $\epsilon_\gamma(m)$ denotes a sign depending on the choice of the representative\footnote{Note that Equation \eqref{signtheta}, which holds for any $m,n \in\Z^g \times \Z^g$, causes the sign ambiguity, as we are focusing on reduced characteristics.} for $\gamma\cdot m$ in $\Z^g\times \Z^g$, and the function $\Phi$ is the product of two factors:
\bes
\Phi(m,\gamma,\tau,z)=\kappa(\gamma)\exp\left\{\frac 12 {}^t z\left[(c\tau+d)^{-1}c\right] z+2\phi_m(\gamma)\right\},
\ees
where
\bes
\phi_m(\gamma)=-\frac 18({}^tm'\,{}^tb\,d\,m' +{}^tm''\,{}^ta\, c\,m''  - 2{}^tm'\,{}^tb\, c\,m'')+ \frac 14 {}^t{\rm diag}(a{}^tb)(d\,m'-c\,m''),
\ees
and, for each $\gamma\in\Gamma_g$, $\kappa(\gamma)$ is an eighth root of the unity, whose sign is determined by choosing the sign of $\det(c\tau+d)^{1/2}$.\\
The general transformation formula for a theta constant is therefore:
\bes \label{eq:transthc}
\theta_{\gamma\cdot m} (\gamma \cdot\tau) = \epsilon_\gamma(m)\kappa(\gamma) \chi_m(\gamma) {\det(c \tau + d)}^{\frac{1}{2}} \theta_m (\tau),
\ees
where:
\be
\label{chi}
\chi_m (\gamma) :=  \Phi (m, \gamma, \tau, 0) = e^{2 \pi i \phi_m(\gamma)}.
\ee
In particular, $\gamma \cdot m = m$ and $\epsilon_\gamma(m)=1$ whenever $\gamma \in \Gamma_g(2)$, because $\Gamma_g(2)$ acts trivially on $\mathcal{C}_g$.\\

\section{Gradients of odd theta functions}

Whenever $n \in O_g$ the theta gradient:
\bes
{\rm grad}^0_z\,\theta_n:={\rm grad}_z\theta_n|_{z=0}=\left(\frac{\partial}{\partial z_1}\theta_n|_{z=0},\ldots ,\frac{\partial}{\partial z_g}\theta_n|_{z=0} \right)
\ees
is easily seen to be a non trivial function. Differentiating both the left and the right term of \eqref{eq:transth} one gets the following transformation law:
\begin{align} \label{eq:transgrad}
&{\rm grad}^0_z\,\theta_{\gamma\cdot n}(\gamma\cdot\tau)=\epsilon_\gamma(n)\kappa(\gamma)\chi_n(\gamma)\det(c\tau+d)^{\frac 12}(c\tau+d){\rm grad}^0_z\theta_n(\tau), \cr
&\forall \gamma\in\Gamma_g,\,\forall \tau\in\hh_g.
\end{align}
The Jacobian determinant of $g$ gradients of odd theta functions in genus $g$ is defined as:
\begin{equation*}
D(n_1,\ldots, n_g) := \pi^{-g}\frac{\partial (\theta_{n_1}, \ldots, \theta_{n_g})}{\partial (z_1,\ldots, z_g)}   \quad \quad \forall n_1,\ldots, n_g \in O_g
\end{equation*}
and transform as follows (cf. \cite{M1T}):
\begin{equation}
\label{Dtransformation}
D(\gamma\cdot n_1,\ldots, \gamma\cdot n_g)(\gamma \cdot \tau)=
\kappa(\gamma)^g\det(c\tau+d)^{g/2+1}\chi_{n_1}(\gamma)\cdot\ldots\cdot\chi_{n_g}(\gamma)D(n_1,\ldots, n_g)(\tau).
\end{equation}
They are non trivial if and only if $n_1,\ldots,n_g$ is such that $e(n_i)e(n_j)e(n_k)e(n_i+n_j+n_k)=-1$ for any $1\leq i<j<k\leq g$ (cf. \cite{Igusa1}). Such $g-$tuple is called an \emph{azygetic $g-$tuple}.

From now on we focus on the genus three case. 
It is classically known that a non singular genus three curve has exactly 28 bitangents which are in bijection with the gradients of odd theta functions. We briefly recall this construction. 
Let $C$ be a non hyperelliptic curve of genus $3$. The period matrix $\tau$ associated with its canonical model naturally defines a complex abelian variety $\Lambda_\tau: = \mathbb{C}^3/(\mathbb{Z}^3+\tau\mathbb{Z}^3)$ whose principal polarization is the Chern class $c(L)$ of the line bundle $L$ associated with the theta divisor $\theta_0$ on $\Lambda_\tau$. Hence $\tau$ actually defines a point in the moduli space $\mathcal{A}_3$.
The Gauss map $\mathcal{G}$ is well defined on the set $\theta_0^r$ of the regular points of the theta divisor:
\begin{align*}
\mathcal{G}:\theta_0^r&\longrightarrow \pp^2 \\
z &\longmapsto {\rm grad}_z\theta_0(\tau,z).
\end{align*}
There are exactly 28 two torsion points of $\Lambda_\tau$ which belong to $\theta_0^r$, namely those of the form $\frac{n'}{2}+\tau\frac{n''}{2}$ with $n',n''\in \mathbb{Z}_2^3$ such that $n=[{}^{n'}_{n''}]\in O_3$.
It is a well known fact that the 28 lines in $\pp^2$ that correspond by duality to the images of these 28 points under the Gauss map are  the bitangents of the curve. Since $\theta_0(\tau,\frac{n'}{2}+\tau\frac{n''}{2})=\theta_n(\tau,0)$, with $n=[{}^{n'}_{n''}]$, the 28 bitangents are actually in bijective correspondence with the gradients of odd theta functions and their equations in $\pp^2$ are:
\be \label{eqbit}
\sum_{i=1}^3 \frac{\partial \theta_{n_j}}{\partial z_i}(\tau,0) z_i=0, \qquad j=1,\ldots,28.
\ee


\section{Proof of the Caporaso-Sernesi theorem}
As in the previous section, let $C$ be a non hyperelliptic curve of genus $3$, then $\theta_m(\tau) \neq 0$ for any $m \in E_3$. Conversely, any point $\tau \in \mathcal{A}_3$ outside the zero locus of the theta constants represents the period matrix of such a curve, uniquely determined up to isomorphisms. 
Let $b_{n_1}(\tau), \dots , b_{n_{28}}(\tau)$ denote the points in $\pp^{2}$ identifying the $28$ bitangents of the curve represented by $\tau$.
Thanks to \eqref{eqbit}, this means that each $b_{n_i}$ is equal to:
\be
\left[\frac{\partial \theta_{n_i}}{\partial z_1}(\tau,0):\frac{\partial \theta_{n_i}}{\partial z_2}(\tau,0):\frac{\partial \theta_{n_i}}{\partial z_3}(\tau,0)\right].
\ee
Theorem \ref{thm:CS} actually claims that whenever $b_{n_i}(\tau)=b_{n_i}(\tau')$ for any $1\leq i\leq 28$, then $\tau$ and $\tau'$ represent the same curve.
Because the group $\Gamma/\Gamma_3(2)\simeq \rm{Sp}(6,\mathbb{Z}_2)$ acts on the 28 odd characteristics, and consequently on the bitangents, just by permutation, it is enough to focus on the Siegel modular variety $\mathcal{A}_3[2]:=\hh_3/\Gamma_3(2)$. More precisely we shall prove the following theorem.
\begin{theorem} \label{thm:2}
Let $\mathcal{A}_3^*[2]$ be the open set of $\mathcal{A}_3[2]$ defined by the condition that no theta constant vanishes, then the map:
\begin{align} \label{map2}
\mathcal{A}_3^*[2]&\rightarrow  {\overbrace{\pp^{2} \times \cdots \times \pp^{2}}^{28 \, \text{times}}} \cr
 \tau &\mapsto [b_{n_1}(\tau):\ldots :b_{n_{28}}(\tau)]
\end{align}
is injective.
\end{theorem}

We resort to Weber's formula \cite{W} to prove Theorem \ref{thm:2}, and consequently Theorem \ref{thm:CS}. We briefly recall it (see \cite{NR} for further details and a new proof).
First we need the following definition.
\begin{deff}
An Aronhold set is a 7-tuple $n_1,\ldots,n_7$ of odd characteristics such that the 8-tuple $n_1,\ldots,n_7,\sum_{i=1}^7 n_i$ is azygetic.
\end{deff}
Whenever $n_1,\ldots,n_7$ is an Aronhold set, then necessarily $m=\sum_{i=1}^7 n_i\in E_3$ (see \cite{DFS} for details). The remaining 21 odd characteristics are of the form $n_{ij}:=m+n_i+n_j$ (cf. \cite{NR}). We denote by $\beta_i(\tau)\in\mathbb{C}^3$ any choice of homogeneous coordinates for the point $b_{n_i}(\tau)$ in $\pp^2$, for $1\leq i \leq 7$, and analogously by $\beta_{ij}(\tau)\in\mathbb{C}^3$ any choice of homogeneous coordinates for the point $b_{n_{ij}}(\tau)$ in $\pp^2$.
\begin{theorem}[Weber's formula]
Let $m_1,m_2\in E_3$ distinct and let $\{n_1,\ldots ,n_7\}$ be an Aronhold set such that $m_1=\sum_{i=1}^7 n_i$ and $n_1+n_2+n_3=m_2$. Let $\tau\in \mathcal{A}_3$ be a point representing the period matrix of a non hyperelliptic curve. Then:
\be \label{weberformula}
\left(\frac{\theta_{m_1}(\tau)}{\theta_{m_2}(\tau)} \right)^4=e(m_1+m_2)\frac{D[\beta_1,\beta_2,\beta_3](\tau) D[\beta_1,\beta_{12},\beta_{13}](\tau) D[\beta_{12},\beta_2,\beta_{23}](\tau) D[\beta_{13},\beta_{23},\beta_3](\tau)}{D[\beta_{23},\beta_{13},\beta_{12}](\tau) D[\beta_{23},\beta_3,\beta_2](\tau) D[\beta_3,\beta_{13},\beta_1](\tau) D[\beta_2,\beta_1,\beta_{12}](\tau)},
\ee
where $D[\beta_{i},\beta_j,\beta_k]$ is the determinant\footnote{The indices $i$, $j$ and $k$ may also denote a double index, for example such as $\beta_{lm}$.} of the $3\times 3$ matrix whose columns are $\beta_i$, $\beta_j$ and $\beta_k$.
\end{theorem}
Now we can prove the statement of Theorem \ref{thm:2}.
\begin{proof}[Proof of Theorem \ref{thm:2}]
First we observe that in the right side of Weber's formula \eqref{weberformula} each $\beta_i$ or $\beta_{ij}$ appears as many times in the numerator as in the denominator, which means that the formula does not depend on the particular choice of the homogeneous  coordinates. Therefore we can choose as coordinates those in equation \eqref{eqbit}. Hence all the determinants in the expression can be replaced with the corresponding Jacobian determinants. This means that if we set:
\be
Q(\tau):=
\frac{D(n_1,n_2,n_3)(\tau) D(n_1,n_{12},n_{13})(\tau) D(n_{12},n_2,n_{23})(\tau) D(n_{13},n_{23},n_3)(\tau)}{D(n_{23},n_{13},n_{12})(\tau) D(n_{23},n_3,n_2)(\tau) D(n_3,n_{13},n_1)(\tau) D(n_2,n_1,n_{12})(\tau)},
\ee
then $Q(\tau)=Q(\tau')$ whenever $\tau,\tau'\in\hh_3$ are the period matrices of two curves $C$ and $C'$ with the same system of bitangents and consequently:
\be
\label{thetaquotients}
\left(\frac{\theta_{m_1}(\tau)}{\theta_{m_2}(\tau)} \right)^4 = \left(\frac{\theta_{m_1}(\tau')}{\theta_{m_2}(\tau')} \right)^4 \qquad \forall\, m_1,m_2\in E_3.
\ee
Now consider the map:
\begin{align*}
 \theta^{(4)}: \mathcal{A}_3[2] &\rightarrow \pp^{35} \\ 
 \quad \tau &\mapsto [\theta_{m_1}^4(\tau):\ldots:\theta_{m_{36}}^4(\tau)].
\end{align*}
For any $\tau \in \mathcal{A}_3^*[2]$, there exists $1\leq i\leq 36$ such that the expression for $ \theta^{(4)}(\tau)$ clearly reduces to:
\begin{align*}
 \theta^{(4)}(\tau) = \left [\left(\frac{\theta_{m_1}(\tau)}{\theta_{m_i}(\tau)}\right)^4:\ldots:\left(\frac{\theta_{m_{i-1}}(\tau)}{\theta_{m_i}(\tau)}\right)^4:1:\left(\frac{\theta_{m_{i+1}}(\tau)}{\theta_{m_i}(\tau)}\right)^4:\ldots:\left(\frac{\theta_{m_{36}}(\tau)}{\theta_{m_i}(\tau)}\right)^4 \right ],
\end{align*}
Therefore, thanks to (\ref{thetaquotients}), two curves with the same system of bitangents have the same image under the  map $ \theta^{(4)}$. The injectivity of the map $ \theta^{(4)}$
(cf. \cite{SM}, Theorem 3) implies that $\tau$ and $\tau'$ are in the same orbit under the action of $\Gamma_3(2)$, which proves the statement of the theorem.
\end{proof}

\begin{remark}
As explained in the introduction, an identity of the form \eqref{ident} exists for any genus $g$, hence there is also an analogous of Weber's formula for higher genus. The technique used in the proof of Theorem \ref{thm:2} can be, therefore, extended to any genus provided that such a generalization of Weber's formula can be written as in \eqref{weberformula} with each odd characteristic appearing in the numerator as many times as in the denominator. This has to be proved yet.
\end{remark}


\end{document}